\documentclass[11pt,a4paper]{article}

\RequirePackage[numbers]{natbib}

\RequirePackage[colorlinks,citecolor=blue,urlcolor=blue]{hyperref}

\usepackage{fullpage,lmodern,bm}
\usepackage[T1]{fontenc}
\usepackage[utf8]{inputenc}                        
\usepackage[Lenny]{fncychap}
\usepackage[usenames,dvipsnames]{xcolor}
\usepackage{amsmath,amssymb,amsthm,thmtools,dsfont,enumerate}
\usepackage{tikz,epigraph,lipsum}
\usetikzlibrary{shapes,arrows}
\usepackage{rotating,lipsum,pifont,cases}
\usepackage{appendix}	
\usepackage{cases}
\usepackage{pdfsync}
\usetikzlibrary{patterns,intersections}

\definecolor{blue0}{RGB}{0,77,153} 
\definecolor{red0}{RGB}{179,0,77} 
\definecolor{green0}{RGB}{134,219,76} 
\definecolor{gray0}{RGB}{84,97,110}

\numberwithin{equation}{section}

\newtheorem{theorem}{Theorem}[section]

\newtheorem{lemma}[theorem]{Lemma}

\newtheorem{remark}[theorem]{Remark}
\newtheorem{example}[theorem]{Example}

\DeclareMathOperator{\supp}{supp}

\newcommand{\C}{{\mathbb C}}

\newcommand{\id}{{\rm id}}

\numberwithin{equation}{section}
\def\vs#1{\vspace{#1mm}}

\def\be{\begin{align}}
\def\ee{\end{align}}
\def\b*{\begin{eqnarray*}}
\def\e*{\end{eqnarray*}}

\def\be{\begin{eqnarray}}
\def\ee{\end{eqnarray}}
\def\beq{\begin{equation}}
\def\eeq{\end{equation}}
\def\b*{\begin{eqnarray*}}
\def\e*{\end{eqnarray*}}
\def\bi{\begin{itemize}}
\def\ei{\end{itemize}}

\def \1{{\bf 1}}

\def\={\;=\;}

 \def\vs#1{\vspace{#1mm}}

\def \E{\mathbb{E}}
\def \F{\mathbb{F}}

\def \P{\mathbb{P}}

\def \R{\mathbb{R}}

\def\Cc{{\cal C}}

\def\Fc{{\cal F}}
\def\Gc{{\cal G}}
\def\Hc{{\cal H}}

\newcommand{\Mid}{{\ \Big|\ }}

\usetikzlibrary{matrix}

\title{Markovian structure of the Volterra Heston model\thanks{We would like to thank Bruno Bouchard and Mathieu Rosenbaum  for very fruitful discussions and comments.}}
\date{\today}

\author{Eduardo Abi Jaber \thanks{Universit\'e Paris-Dauphine, PSL Research University, CNRS, UMR [7534], CEREMADE, 75016 Paris, France and AXA Investment Managers,  Multi Asset Client Solutions, Quantitative Research, 6 place de la Pyramide, 92908 Paris - La D\'efense, France, abijaber@ceremade.dauphine.fr.}
\and Omar El Euch\thanks{CMAP, Ecole Polytechnique, Palaiseau, France, omar.el-euch@polytechnique.edu}}

\begin{document}

	\maketitle

	\begin{abstract}

		We characterize the  {Markovian and affine structure} of the {Volterra Heston model} in terms of an infinite-dimensional adjusted forward process and specify its state space. More precisely, we show that it satisfies a {stochastic partial differential equation}  and displays an exponentially-affine characteristic functional. As an application, we deduce an existence and uniqueness result for a Banach-space valued  square-root process and provide its state space. This leads to another   representation of the Volterra Heston model together with its Fourier-Laplace transform in terms of this possibly infinite system of affine diffusions.  \vs1

		
		
	\end{abstract}



\section{Introduction}
{The {Volterra Heston model}   is defined by the following dynamics}
\begin{align}
dS_t &= S_t \sqrt{V_t} dB_t, \quad S_0>0, \label{E:HestonlogSg}\\
V_t  &= g_0(t) + \int_0^t K(t-s) \left( - \lambda V_s ds + \nu \sqrt{V_s}dW_s \right) \label{E:HestonVg},
\end{align}
with $K \in L^2_{\rm loc}(\R_+,\R)$, $g_0: \R_+ \to \R$, $\lambda,\nu \in\R_+ $ and  $B = \rho W + \sqrt{1-\rho^2} W^\perp $ such that $(W, W^\perp)$ is a two-dimensional Brownian motion and $\rho \in [-1,1]$. It has been introduced in \cite{ALP17} for the purpose of financial modeling   following the literature on so-called rough volatility models \cite{gatheral2014volatility}. Hence $S_t$ typically represents a stock price at time $t$ with instantaneous stochastic variance $V_t$.  \vs1


This model nests as special cases the Heston model for $K\equiv 1$, and the rough Heston model of \cite{euch2016characteristic}, obtained by setting  $K(t)=\frac{t^{\alpha-1} }{\Gamma(\alpha)}$  for $\alpha \in (\frac 12, 1)$ and 
\begin{equation} \label{particular_g}
g_0(t) = V_0 + \int_0^{t} K(s) \lambda \theta ds,       \quad t\geq 0, \quad   \mbox{for some }V_0,  \theta \geq 0, 
\end{equation}
so that the only {model} parameters are $V_0,\theta,\lambda, \rho, \nu,\alpha$. Recall that the rough Heston model does not only fit remarkably well historical and implied volatilities of the market, but also enjoys a semi-closed formula for the characteristic function of the log-price in terms of {a} solution of a deterministic {Riccati-Volterra integral equation}. \vs1

In \cite{euch2017perfect}, the authors highlight the crucial role of  \eqref{particular_g}  in the design of hedging strategies for the rough Heston model. Here  
 we consider more general input curves $g_0$. Our motivation is twofold. In practice, the function $g_0$ is intimately linked to the forward variance curve $(\E[V_t])_{t \geq 0}$. More precisely, taking the expectation in \eqref{E:HestonVg} leads to the following relation
		$$ \E[V_t] + \lambda \int_0^t K(t-s) \E[V_s]ds =  g_0(t)  , \quad t \geq 0. $$  
	Thus, allowing for more general input curves  $g_0$ leads to more consistency with the market forward variance curve.  From a mathematical perspective,   this enables us to understand the general picture behind the  Markovian and affine nature of the  {Volterra Heston model \eqref{E:HestonlogSg}-\eqref{E:HestonVg}.\vs1

More precisely, adapting the methods of \cite{ALP17}, we provide a set of {admissible input curves} $\Gc_K$ defined in \eqref{E:DomainG} such that \eqref{E:HestonlogSg}-\eqref{E:HestonVg} admits a unique  $\R_+^2$-valued weak solution for any $g_0 \in \Gc_K$. In particular, we show that the Fourier-Laplace transform of $(\log S, V)$ is exponentially affine in $(\log S_0, g_0)$. 
Then we prove that, conditional on $\Fc_t$, the shifted {Volterra Heston model} $(S_{t+\cdot},V_{t+\cdot})$ still  has the same dynamics as in \eqref{E:HestonlogSg}-\eqref{E:HestonVg} {provided that} $g_0$ is replaced by  the following adjusted forward process
\begin{align}\label{E:gt0}
g_{t}(x) =\E\left[V_{t+x} + \lambda \int_0^x K(x-s)V_{t+s}ds \Mid \Fc_t\right], \quad x\geq 0 .
\end{align}  
 This leads to our main result which states that $\Gc_K$ is {stochastically invariant} with respect to the family $(g_t)_{t\geq 0}$. In other words, if we start from an  	{initial admissible input curve} $g_0 \in \Gc_K$, then $g_t$ belongs to $\Gc_K$, for all $t \geq 0$, see Theorem  \ref{T:condLaw0}. This in turn enables us to characterize the Markovian structure of $(S,V)$ in terms of the stock price and the adjusted forward process $(g_t)_{t \geq 0}$. Furthermore,  $(g_t)_{t\geq 0}$  can be realized as  the unique $\Gc_K$-valued  mild solution of the following {stochastic partial differential equation} of Heath–Jarrow–Morton-type
 \begin{align*}
 dg_t(x)  = \left( \frac{d}{dx} g_t(x)  - \lambda K(x) g_t(0) \right) dt +   K(x) \nu \sqrt{g_t(0)} dW_t,  \quad g_0 \in \Gc_K,
 \end{align*}
{and displays an affine characteristic functional.} \vs1
 
As an application, we establish the existence and uniqueness of a {Banach-space valued square-root process} and provide its state space.  This leads to another representation of $(V_t,g_t)_{t \geq 0}$. Moreover, the Fourier-Laplace transform of $(\log S, V)$ is shown to be an exponential affine functional of this process. These results are  in the spirit of the Markovian representation of fractional Brownian motion, see \cite{CC98,HS15}.\vs1

The paper is organized as follows. 	In Section \ref{S:extension}, we prove weak existence and uniqueness for the {Volterra Heston model} and provide its Fourier-Laplace transform.  Section  \ref{S:markovian} characterizes the Markovian structure in terms of the adjusted forward variance process.  Section \ref{S:cm} establishes the existence and uniqueness of a Banach-space valued square-root process and provides  the link with the Volterra framework.   In Appendix \ref{A:sve} we derive general existence results for {stochastic Volterra equations}. Finally, for the convenience of the reader we recall in Appendix \ref{A:conv} the framework and notations regarding  stochastic convolutions as  in \cite{ALP17}. \vs1

\textbf{Notations {:}} Elements of $\C^m$ are viewed as column vectors, while elements of the dual space $(\C^m)^*$ are viewed as row vectors.  
For $h\ge0$,  $\Delta_h$ denotes the shift operator, i.e.~
 $\Delta_h f(t) = f(t+h).$ If the function $f$ on $\R_+$ is right-continuous and of locally bounded variation, the measure induced by its distribution derivative is denoted $df$, so that $f(t) = f(0) + \int_{[0,t]} df(s)$ for all $t\ge0$.  Finally, we use the notation $*$ for the convolution operation, we refer to Appendix \ref{A:conv} for more details.

\section{{Existence and uniqueness} of the Volterra Heston model}\label{S:extension}
	We study in this section  the existence and uniqueness of the {Volterra Heston model} given by \eqref{E:HestonlogSg}-\eqref{E:HestonVg} allowing for arbitrary curves $g_0$ as input. 
When $g_0$ is given by \eqref{particular_g},  \cite[Theorem 7.1(i)]{ALP17} provides the existence of a $\R^2_+$-valued weak solution to \eqref{E:HestonlogSg}-\eqref{E:HestonVg} under the following mild assumptions on $K$:
\begin{equation} \label{K_gamma}\tag{$H_0$}
\begin{minipage}[c][1.5em][c]{.8\textwidth}
\begin{center}
$K\in L^2_{\rm loc}(\R_+,\R)$, and there is $\gamma\in(0,2]$ such that $\int_0^h K(t)^2dt = O(h^\gamma)$
and $\int_0^T (K(t+h)-K(t))^2 dt = O(h^\gamma)$ for every $T<\infty$,
\end{center}
\end{minipage}
\end{equation}
\begin{equation} \label{eq:K orthant}\tag{$H_1$}
\begin{minipage}[c][1.5em][c]{.82\textwidth}
\begin{center}
$K$ is nonnegative,  not identically zero, non-increasing and continuous on $(0, \infty)$, and its {\em resolvent of the first kind} $L$ is nonnegative and non-increasing in the sense that $s \to  L([s,s+t])$ is non-increasing for all $t\ge0$. 
\end{center}
\end{minipage}
\footnote{We refer to Appendix~\ref{A:conv} for the definition of the {\em resolvent of the first kind} and some of its properties.}
\end{equation}
\vs1

We show in Theorem \ref{T:VolterraHestongExistence} below   that weak existence in $\R_+^2$ continue to hold for \eqref{E:HestonlogSg}-\eqref{E:HestonVg} for a wider class of {admissible input curves} $g_0$. {Since} $S$ is  determined by $V$, it suffices to study the Volterra square-root equation \eqref{E:HestonVg}.   Theorem~\ref{T:existenceg}\eqref{T:existenceg2} in the Appendix guarantees  the existence of an unsconstrained continuous weak solution $V$ to the  following modified equation 
\begin{align}\label{E:modifiedVg}
V_t = g_0(t) +  \int_0^t K(t-s) \left( -\lambda V_s ds + \nu \sqrt{V_s^+} dW_s\right),
\end{align}
for any locally H\"older continuous function $g_0$, where $x^+: x\to\max(0,x)$. 
{Clearly,} one  needs to impose additional assumptions on $g_0$ to ensure the nonnegativity of  $V$ and  drop the positive part in \eqref{E:modifiedVg} so that $V$ solves \eqref{E:HestonVg}. Hence, weak existence of a nonnegative solution to \eqref{E:HestonVg} boils down to finding  a set $\Gc_K$ of {admissible input curves} $g_0$  such that any solution  $V$  to \eqref{E:modifiedVg} is nonnegative. \vs1

\noindent To get a taste of the {admissible set} $\Gc_K$, we start by assuming that $g_0$ and $K$ are continuously differentiable on $[0,\infty)$.  In that case, $V$ is a semimartingale such that 
\begin{align}\label{E:modifieddiffusion}
dV_t  = \left( g_0'(t)  + (K'*dZ)_t -  K(0)\lambda V_t  \right)dt + K(0) \nu \sqrt{V_t^+} dW_t,
\end{align}
where $Z=\int_0^{\cdot}(-\lambda V_s ds + \nu \sqrt{V_s^+}dW_s)$. Relying on Lemma \ref{L:ZX} in the Appendix\footnote{Under   \eqref{eq:K orthant} one can show that   $K'*L$  is right-continuous, non-decreasing and of locally bounded variation (as in Remark~\ref{DeltaK} in the Appendix), thus the associated measure $d(K'*L)$ is well defined.}, we have 
$$ K' = (K'*L)(0) K + d(K'*L)*K,$$
so that  $K'*dZ$ can be expressed as a functional of $(V,g_0)$ as follows 
\begin{align}\label{E:K'dz}
 K'*dZ =  (K'*L)(0) (V - g_0) + d(K'*L)*(V-g_0).
\end{align}
Since $V_0=g_0(0)$, it is straightforward that $g_0(0)$ should be nonnegative. Now, assume that $V$ hits zero for the first time at $\tau \geq 0$.
After plugging \eqref{E:K'dz} in the drift of \eqref{E:modifieddiffusion},  a first-order Euler scheme leads to the formal approximation 
\begin{align*}
 V_{\tau+h} &\approx \left (  g_0'(\tau) - (K'*L)(0)g_0(\tau)  - (d(K'*L)*g_0)(\tau)   +  (d(K'*L)*V)_{\tau} \right) h,
\end{align*}
for  small $h \geq 0$. Since $K'*L$ is non-decreasing and $V \geq 0$ on $[0,\tau]$, it follows that $ (d(K'*L)*V)_{\tau}  \geq 0$ yielding the nonnegativity of $V_{\tau + h}$ if we impose the following additional condition 
\begin{align*}
g_0'- (K'*L)(0)g_0  - d(K'*L)*g_0  \geq 0.
\end{align*}

{In the general case,} $V$ is not necessarily a semimartingale, and a delicate analysis should be carried on the integral equation \eqref{E:modifiedVg} instead of the infinitesimal version \eqref{E:modifieddiffusion}. This suggests  that the infinitesimal derivative operator should be replaced by the semigroup operator of right shifts leading to the following condition  on $g_0$
\begin{equation} \label{Croissance}
\Delta_h g_0 - (\Delta_hK * L)(0) g_0 - d(\Delta_hK * L) * g_0 \geq 0, \quad  h \geq 0, \footnote{
Recall that under  \eqref{eq:K orthant} one can show that   $\Delta_hK*L$  is right-continuous and of locally bounded variation (see Remark~\ref{DeltaK} in the Appendix), thus the associated measure $d(\Delta_hK*L)$ is well defined.}
\end{equation}
and to the following definition of the set $\Gc_K$ of {admissible input curves}
\begin{align}\label{E:DomainG}
{\Gc_K = \left\{ g_0 \in \Hc^{\gamma/2}  \mbox{ satisfying } \eqref{Croissance} \mbox{ and } g_0(0) \geq 0  \right\},}
\end{align} 
where {$\Hc^{\alpha}=\{ g_0:\R_+ \to \R,  \mbox{ locally H\"older continuous of any order strictly smaller than $\alpha$}\}.$} Recall that $\gamma$ is the exponent associated with $K$ in~\eqref{K_gamma}. \vs1

The following theorem establishes the existence of a $\R_+^2$-valued weak continuous solution to \eqref{E:HestonlogSg}-\eqref{E:HestonVg} 
 on some filtered probability space $(\Omega,\Fc,\F=(\Fc_t)_{t \geq 0},\P)$ for any admissible input curve $g_0 \in \Gc_K$. Since $S$ is determined by $V$, the proof follows  directly from Theorems~\ref{T:existenceg}-\ref{T:existence orthant}.
\begin{theorem} \label{T:VolterraHestongExistence}
Assume that $K$ satisfies \eqref{K_gamma}-\eqref{eq:K orthant}. Then, the stochastic Volterra equation \eqref{E:HestonlogSg}-\eqref{E:HestonVg} has a  $\R_+^2$-valued continuous weak solution $( S,V)$ for any positive initial condition $S_0$ and any admissible input curve $g_0 \in \Gc_K$. Furthermore, the paths of $V$ are {locally} H\"older continuous of any order strictly smaller than $\gamma/2$ and 
 		\begin{align}\label{E:momentestimate}
		 \sup_{t \leq T}\E [|V_t|^p] < \infty , \quad p>0, \quad T>0. 
		 \end{align}
\end{theorem}

\begin{example}\label{Ex:Gcal}The following classes of functions  belong to $\Gc_K$. 
\begin{enumerate}[(i)]
\item  	 $g \in \Hc^{\gamma/2}$ non-decreasing such that $g(0)\geq 0$. Since $K$ is non-increasing and  $L$ is nonnegative{,} we have $0 \leq \Delta_h K*L \leq 1$ for all $h\geq 0 $  (see the proof of \cite[Theorem 3.5]{ALP17}) yielding, for all $t,h \geq 0$, that $\Delta_h g(t) - (\Delta_hK * L)(0) g(t) - (d(\Delta_hK * L) *g)(t)$ is equal to 
\begin{align*}
{ \int_0^t (g(t)-g(t-s)) (\Delta_hK*L)(ds) + g(t+h) - g(t) + g(t) (1-(\Delta_hK*L)(t)) \geq 0}
\end{align*}
\item  \label{Ex:gform2}  $g = V_0 + K*\theta $, with $V_0 \geq 0$ and $\theta \in L^2_{loc}(\mathbb R_+, \mathbb R) $ such that $\theta(s) ds+ V_0 L(ds)$ is a nonnegative measure. First, 	 $g \in \Hc^{\gamma/2}$  due to \eqref{K_gamma} and the Cauchy-Schwarz inequality
\begin{align*}
(g(t+h) - g(t))^2 \leq 2 \left( \int_0^t (K(s+h) - K(s))^2 ds + \int_0^h K(s)^2 ds \right) \int_0^{t+h} \theta(s)^2 ds.  
\end{align*}
Moreover, $g(0)=V_0 \geq 0$  and
\begin{align}\label{cal} 
\Delta_hg - (\Delta_hK * L)(0) g- d(\Delta_hK * L) * g  
\end{align}
is equal to 
$$ V_0 (1-\Delta_hK*L) + {\Delta_h (K*\theta)} - (\Delta_hK*L)(0) K*\theta - d(\Delta_hK*L)*K*\theta.$$
 \eqref{Croissance} now follows  from  Lemma \ref{thelemma} with $F=\Delta_h K$, after noticing that  \eqref{cal} becomes
$$ \Delta_h(K*(V_0 L + \theta)) - \Delta_hK*(V_0 L + \theta) = \int_{\cdot}^{\cdot+h}K(\cdot+h-s)(V_0 L(ds) + \theta(s) ds) \geq 0.$$
\end{enumerate}
\end{example}

We now tackle the weak uniqueness of \eqref{E:HestonlogSg}-\eqref{E:HestonVg} by characterizing the Fourier-Laplace transform of the process $X=(\log S, V)$. Indeed,  when $g_0$ is of the form \eqref{particular_g}, $X$ is a two-dimensional affine Volterra process in the sense of \cite[Definition 4.1]{ALP17}. For this particular $g_0$, \cite[Theorem 7.1(ii)]{ALP17} provides the exponential-affine transform formula  
\begin{align}\label{E:Y0g}
\E[ \exp (u X_T +  (f * X)_T)] =\exp \bigg(  \psi_1(T) \log S_0 + u_2 g_0(T) + \int_0^T F(\psi_1,\psi_2)(s)g_0(T-s)ds \bigg)
\end{align}
for suitable $u \in \mathbb (C^2)^*$ and $f \in L^1([0,T], \mathbb (C^2)^*)$ with $T>0$, 
where  $\psi=(\psi_1,\psi_2)$ solves the following system of Riccati-Volterra equations 
\begin{align}
\psi_1 &= u_1+1*f_1,\label{E:RicHeston0} \\
\psi_2 &=u_2K +K*F(\psi_1,\psi_2), \label{E:RicHeston1}
\end{align} 
with 
\begin{align}\label{E:phichi-2}
F(\psi_1,\psi_2) = f_2+\frac12\left( \psi_1^2-\psi_1\right) +(\rho \nu \psi_1- \lambda)  \psi_2 + \frac{\nu^2}{2} \psi^2_2.
\end{align}
	
A straightforward adaptation of \cite[Theorems~4.3 and 7.1]{ALP17}   shows that the  affine transform \eqref{E:Y0g} carries over for any {admissible input curve} $g_0 \in \Gc_K$ with  the same Riccati equations \eqref{E:RicHeston0}-\eqref{E:RicHeston1}.

	\begin{theorem} \label{T:VolterraHestong}
		Assume that $K$ satisfies \eqref{K_gamma} and that the shifted kernels $\Delta_h K$ satisfy \eqref{eq:K orthant} for all $h\in[0,1]$. Fix {$g_0 \in \Gc_K,  S_0 > 0$ and denote  by $(S, V)$  a $\R^2_+$-valued continuous weak solution to  \eqref{E:HestonlogSg}-\eqref{E:HestonVg}. For any $u\in(\mathbb C^2)^*$ and $f\in L^1_{\rm loc}(\R_+,(\mathbb C^2)^*))$ such that
		\begin{align}\label{E:condcoeffrealpart}
		\text{${\rm Re\,} \psi_1 \in[0,1]$, ${\rm Re\,} u_2 \le0$ and ${\rm Re\,} f_2\le0$,}
		\end{align}
		with $\psi_1$  given by \eqref{E:RicHeston0}, the Riccati--Volterra equation \eqref{E:RicHeston1} admits a unique global solution $\psi_2\in L^2_{\rm loc}(\R_+,\mathbb C^*)$. Moreover, the exponential-affine transform  \eqref{E:Y0g} is satisfied. In particular, weak uniqueness holds for \eqref{E:HestonlogSg}-\eqref{E:HestonVg}.}
	\end{theorem}

\section{Markovian structure}\label{S:markovian}
		
 Using the same methodology as in \cite{euch2017perfect}, we characterize the Markovian structure of the {Volterra Heston model} \eqref{E:HestonlogSg}-\eqref{E:HestonVg} in terms of the $\F$-adapted infinite-dimensional adjusted  forward curve $(g_t)_{t \geq 0}$ given by \eqref{E:gt0} which is well defined thanks to~\eqref{E:momentestimate}.  Furthermore, we prove that the set $\Gc_K$ is stochastically invariant with respect to  $(g_t)_{t \geq 0}$.  

\begin{theorem} \label{T:condLaw0}
	Under the assumptions of Theorem \ref{T:VolterraHestongExistence}, fix $g_0 \in \Gc_K$. Denote by $(S,V)$ the unique solution to \eqref{E:HestonlogSg}-\eqref{E:HestonVg} and by $(g_t)_{t \geq 0}$ the process defined by~\eqref{E:gt0}. Then, $(S^{t_0}, V^{t_0} )$ satisfies
 \begin{align*}
dS_t^{t_0} &= S_t^{t_0} \sqrt{V_t^{t_0}} dB_t^{t_0}, \quad S_0^{t_0} = S_{t_0}, \\
V_t^{t_0}&= g_{t_0}(t) +  \int_0^{t} K(t-s) \left( -\lambda V_s^{t_0} ds + \nu \sqrt{V_s^{t_0}}dW_s^{t_0}\right),
\end{align*}
where $(B^{t_0}, W^{t_0}) = (B_{t_0+\cdot}-B_{t_0},W_{t_0+\cdot}-W_{t_0})$ are two Brownian motions independent of ${\cal F}_{t_0}$ such that $d \langle B^{t_0},W^{t_0}\rangle_t = \rho dt$.
Moreover, $\Gc_K$ is stochastically invariant with respect to $(g_{t})_{{t} \geq 0}$, that is  $$g_{t} \in \Gc_K, \quad   t \geq 0.$$
\end{theorem}

\begin{proof}
 The part for $V^{t_0}$ is immediate after observing that  
\begin{align}\label{Eq:SPDEg}
g_{t_0}(t) = g_0({t_0 + t}) - \int_0^{t_0} K(t+t_0-s) \lambda V_s ds +   \int_0^{t_0} K(t+t_0-s) \nu \sqrt{V_s} dW_s,  
\end{align}
for all $t_0,t,h \geq 0$. The part for $S^{t_0}$ is straightforward. We move to proving the claimed invariance.  Fix $t_0,t,h \geq 0$ and define $Z=\int_0^{\cdot} (-\lambda V_s ds + \nu \sqrt{V_s}dW_s)$.  By Lemma \ref{thelemma} and Remark \ref{DeltaK} in the Appendix,
\begin{align}
\Delta_h K&= (\Delta_h K*L)(0) K + d(\Delta_h K*L)*K,  \label{E:tempDhK} 
\end{align}
so that 
\begin{align*}
(\Delta_h K *dZ) &= (\Delta_h K*L)(0) (V - g_0) + d(\Delta_h K*L)*(V-g_0). 
\end{align*}
Hence,  
\begin{align*}
 V^{t_0}_{t+h}  &= g_0 (t_0 +t+h) + \left( \Delta_h K * dZ \right)_{t_0+ t}  + \int_0^h K(h-s) dZ_{t_0 + t+s}\\
&= g_0 (t_0 +t+h) + (\Delta_h K*L)(0) (V^{t_0}_t - g_0(t_0+t)) \\
& \quad + \left(d(\Delta_h K*L)*(V-g_0)\right)_{t_0+ t}  + \int_0^h K(h-s) dZ_{t_0 + t+s} \\
&=  g_0 (t_0 +t+h) - (\Delta_h K*L)(0) g_0(t_0+t) -\left(d(\Delta_h K*L)*g_0\right)({t_0+ t})  \\
& \quad + (\Delta_h K*L)(0) V^{t_0}_t +  \left(d(\Delta_h K*L)*V\right)_{t_0+ t}  + \int_0^h K(h-s) dZ_{t_0 + t+s}\\
& \geq (\Delta_h K*L)(0) V^{t_0}_t  +  \left(d(\Delta_h K*L)*V\right)_{t_0+ t}  - \int_0^h K(h-s) \lambda V^{t_0}_{t+s}ds \\
&\quad + \int_0^h K(h-s) \nu \sqrt{V^{t_0}_{t+s}} dW^{t_0}_{t+s},
\end{align*}
since $g_0 \in \Gc_K$. We  now prove \eqref{Croissance}. Set $G^{t_0}_h = \Delta_h g_{t_0} - (\Delta_hK * L)(0) g_{t_0} - d(\Delta_hK * L) * g_{t_0} $.  {The previous inequality combined with  \eqref{E:gt0} yields }
\begin{align*}
G^{t_0}_h(t) &=   \E\left[ V^{t_0}_{t+h} +(\lambda K* V^{t_0})_{t+h}     - (\Delta_hK * L)(0)(V^{t_0}_t +  (\lambda K* V^{t_0})_{t} )   \Mid \Fc_{t_0}\right] \\
 &\quad  - \E\left[  \left(d(\Delta_hK * L) * (V^{t_0} +\lambda K* V^{t_0} ) \right)_t       \Mid \Fc_{t_0}\right] \\
& \geq  \E\left[   \left(d(\Delta_h K*L)*V\right)_{t_0+ t}  -\left(d(\Delta_hK * L) * V^{t_0} \right)_t   - \int_0^h K(h-s) \lambda V^{t_0}_{t+s}ds \Mid \Fc_{t_0}\right] \\
& \quad +  \E\left[  (\lambda K* V^{t_0})_{t+h}  - \left( \left((\Delta_hK * L)(0) K +d(\Delta_hK * L) * K\right)    * \lambda V^{t_0} \right)_{t}    \Mid \Fc_{t_0}\right].
\end{align*}
Relying on \eqref{E:tempDhK}, we deduce 
\begin{align*}
G^{t_0}_h(t) &\geq   \E\left[    \int_t^{t_ 0+t} (d(\Delta_h K*L))(ds) V_{t_0+t-s} - \int_0^h K(h-s) \lambda V^{t_0}_{t+s}ds \Mid\Fc_{t_0}\right] \\
& \quad +  \E\left[  \int_t^{t+h} K(t+h-s) \lambda {V^{t_0}_{s}}ds   \Mid\Fc_{t_0}\right]\\
&= \E\left[    \int_t^{t_ 0+t} (d(\Delta_h K*L))(ds) V_{t_0+t-s}  \Mid\Fc_{t_0}\right].
\end{align*}
Hence  \eqref{Croissance} holds for $g_{t_0}$, since $V\geq0$ and $d(\Delta_h K*L)$ is a nonnegative measure, see Remark \ref{DeltaK}. Finally, by adapting the proof of \cite[Lemma~2.4]{ALP17}, we can show that for any {$ p>1,  \epsilon >0$ and $T >0$, there exists a positive constant $C_1$ such that 
$$\E\left[|V_{t+h} -V_t|^p \right] \leq C_1 h^{p(\gamma/2-\epsilon)}, \quad t,h \geq 0, \; t+h \leq T+t_0, $$
Relying on \eqref{K_gamma}, \eqref{E:gt0} and Jensen inequality, there exists a positive constant $C_2$ such that
$$\E\left[|g_{t_0}(t+h) -g_{t_0}(t)|^p \right] \leq C_2 h^{p(\gamma/2-\epsilon)}, \quad t,h \geq 0, \; t+h \leq T, $$
By  Kolmogorov continuity criterion,   $g_{t_0} \in \Hc^{\gamma/2}$ so that  {$g_{t_0} \in \Gc_K$ since  $g_{t_0}(0) = V_{t_0} \geq 0$. } }
\end{proof}

Theorem \ref{T:condLaw0} highlights that $V$ is Markovian in the state variable $(g_t)_{t \geq 0}$. Indeed, conditional on $\Fc_t$ for some $t\geq 0$, the shifted {Volterra Heston model} $(S^t,V^t)$ can be started afresh from $(S_t,g_t)$ with the same dynamics as in \eqref{E:HestonlogSg}-\eqref{E:HestonVg}. Notice that $g_t$ is again an {admissible input curve} belonging to $\Gc_K$. Therefore, applying Theorems \ref{T:VolterraHestongExistence} and \ref{T:VolterraHestong}  with $(S^t,V^t,g_t)$ yields that the conditional Fourier-Laplace transform of $X=(\log S, V)$ is exponentially affine in $(\log S_t,g_t)$:
\begin{align}\label{E:charwithgt}
\E\left[ \exp(uX_T + (f*X)_T) \Mid \Fc_{t} \right] = \exp\left(\psi_1(T-t)  \log S_{t}     + (u_2g_{t} + F(\psi_1,\psi_2)*g_{t})(T-t)\right),
\end{align}
for all $t \leq T$, where $F$ is given by \eqref{E:phichi-2}, under the standing assumptions of Theorem \ref{T:VolterraHestong}. \vs1

Moreover, it follows from~\eqref{Eq:SPDEg} and the fact that $g_{\cdot}(0)= V$  that the process $(g_t)_{t \geq 0}$ solves 
		\begin{align}\label{E:spdemild}
			g_{t}(x) = \Delta_t g_0(x) + \int_0^{t} \Delta_{t-s}  \left(- \lambda K g_s(0)\right)(x) ds + \int_0^t \Delta_{t-s}  \left( K \nu \sqrt{g_s(0)} \right)(x) dW_s.
		\end{align}
		Recalling that  $(\Delta_t)_{t \geq 0}$ is the semigroup of right shifts,	\eqref{E:spdemild} can be seen as a  $\Gc_K$-valued mild solution of the following  Heath–Jarrow–Morton-type \textit{stochastic partial differential equation}
		\begin{align}\label{E:SPDE}
		dg_t(x)  = \left( \frac{d}{dx} g_t(x)  - \lambda K(x) g_t(0) \right) dt +   K(x) \nu \sqrt{g_t(0)} dW_t,  \quad g_0 \in \Gc_K.
		\end{align}
		The following proposition provides the characteristic functional of  $(g_t)_{t \geq 0}$ leading to  the strong Markov property of $(g_t)_{t \geq 0}$. Define $\langle g,h  \rangle = \int_{\R_+} g(x) h(x) dx$, for suitable functions $f$ and $g$.
		\begin{theorem}\label{T:charfunctionalg} Under the assumptions of Theorem \ref{T:VolterraHestong}.
		   Let $h \in \Cc_c^{\infty}(\R_+)$  and $g_0 \in \Gc_K$. Then,
		\begin{align}\label{E:charfunctionalg}
		\E\left[  \exp\left(   {\rm i} \langle g_t, h \rangle  \right) \right] = \exp \left( \langle  H_{t}  , g_0 \rangle \right), \quad t \geq 0,
		\end{align}
		where $H$ solves 
		\begin{equation}\label{E:Htransform}
		H_t(x) =  {\rm i} h(x-t)  \mathds{1}_{\{x > t\}} + \mathds{1}_{\{x \leq t\}} \bigg( -  \lambda \langle H_{t-x}, K \rangle + \frac{\nu^2}2 \langle H_{t-x}, K \rangle^2\bigg), \quad {t,x \geq 0}.
		\end{equation}
		In particular, weak uniqueness holds for \eqref{E:spdemild} and $(g_t)_{t \geq 0}$ is a strong Markov process on $\Gc_K$.
	   \end{theorem}
		\begin{proof} Consider  $\widetilde S_t= 1 + \int_0^t \widetilde S_u \sqrt{V_u}dW_u$, for all  $t \geq 0.$ Then, $(\widetilde S, V)$ is a {Volterra Heston model} of the form \eqref{E:HestonlogSg}-\eqref{E:HestonVg} with $\rho=1$ and $\widetilde S_0=1$.
			Fix $t \geq 0$, $\langle g_t,h \rangle$ is well defined since $x \to g_t(x)$ is continuous. It follows from~\eqref{Eq:SPDEg} together with stochastic Fubini theorem, see \cite[Theorem~2.2]{V:12}, which is justified by \eqref{E:momentestimate},  that 
			\begin{align*}
			\langle  g_t,h\rangle &= \langle  g_0(t+\cdot),h \rangle + \left(\frac{\nu}2 - \lambda\right) \int_0^t \langle   K(t-s+ \cdot),h \rangle  V_s ds   + \nu \int_0^t    \langle  K(t-s+\cdot),h\rangle d(\log \widetilde S)_s \\
			 &= \langle  g_0,h(-t+\cdot) \rangle + \left(\frac{\nu}2 - \lambda\right) \int_0^t \langle   K,h(s-t+ \cdot) \rangle  V_s ds \\
			 &\quad  + \nu\langle K, h\rangle \log \widetilde S_t  - \nu \int_0^t    \langle  K,h'(s-t+\cdot)\rangle \log \widetilde S_sds,
			\end{align*}
			where the last identity follows from an integration by parts. 
			Hence,     setting 
			\begin{align*}
			u_2 &= 0,  \quad u_1 =  {\rm i} \nu \langle K , h \rangle , \quad  f_1(t) =   -{\rm i}\nu \langle K, h'(-t+\cdot) \rangle, \\
			\psi_1(t) &= u_1 + (1*f_1)(t) = {\rm i }\nu \langle K(t+\cdot), h \rangle , \\
			f_2 (t) &=  {\rm i }(\frac{\nu}2 - \lambda)  \langle K(t+\cdot), h \rangle,  \quad  \psi_2 = K*F(\psi_1, \psi_2),
			\end{align*}
	with $F$ as in \eqref{E:phichi-2}, the characteristic functional follows from Theorem~\ref{T:VolterraHestong}
			\begin{align*}
			\E\left[  \exp\left(   {\rm i} \langle g_t, h \rangle  \right)  \right] &= e^{ {\rm i} \langle h(-t+ \cdot),g_0 \rangle   } \E\left[  \exp\left(  u_1 \log \widetilde S_t + (f_1 * \log \widetilde S)_t + (f_2*V)_t \right)  \right] = \exp \left( \langle  H_t  , g_0 \rangle \right)
			\end{align*}
			where 	$$H_t(x) = h(x-t) \mathds{1}_{\{x > t\}} + \mathds{1}_{\{ 0 \leq  x \leq t\}} F(\psi_1,\psi_2)(t-x), \quad  x \geq 0,$$
		and \eqref{E:phichi-2} reads
		\begin{align}\label{E:chi'tilde}
		 F(\psi_1,\psi_2)(t)&= -\lambda  \langle K(t+\cdot), h \rangle    + \frac{\nu^2}2  \langle K(t+\cdot), h \rangle ^2  \nonumber \\
		 &\quad + ( \nu^2  \langle K(t+\cdot), h \rangle  - \lambda)\psi_2(t) + \frac{\nu^2}2 \psi_2(t)^2.
		\end{align}
	Now observe that 
	\begin{align*}
	 \langle H_t,K\rangle  =  \langle  h(-t+\cdot),K \rangle + \int_0^t  F(\psi_1,\psi_2)(t-x) K(x) dx   =  \langle  h,K(t+\cdot) \rangle +\psi_2(t) . 	
	\end{align*}	
			Hence, after plugging  $\psi_2(t) =  \langle H_t,K\rangle -  \langle h, K(t+ \cdot) \rangle$ back in~\eqref{E:chi'tilde} we get that
			$$  F(\psi_1,\psi_2)(t) = -\lambda  \langle H_t,K\rangle + \frac{\nu^2}2  \langle H_t,K\rangle^2,$$
			yielding \eqref{E:Htransform}.  Weak uniqueness now follows by standard arguments. In fact, thanks to \eqref{E:momentestimate} and stochastic Fubini theorem, $(g_t)_{t \geq 0}$ solves \eqref{E:SPDE} in the weak sense, that is 
			$$ \langle g_t,h \rangle  = \langle g_0,h \rangle + \int_0^t \left(\langle g_s,-h' \rangle  -\lambda   \langle K,h \rangle g_s(0)  \right)ds  + \int_0^t  \nu\langle K,h \rangle \sqrt{g_s(0)}dW_s , \quad h\in \Cc^{\infty}_c(\R).$$ Therefore, combined with Theorem \ref{T:condLaw0}, $(g_t)_{t \geq 0}$ solves a martingale problem on $\Gc_K$. In addition,  \eqref{E:charfunctionalg} yields uniqueness of the one-dimensional distributions which is enough to get   weak uniqueness  for \eqref{E:spdemild} and the strong Markov property by \cite[Theorem~4.4.2]{EK86}. 
		\end{proof}
		
		
			We  notice that \eqref{E:charfunctionalg}-\eqref{E:Htransform} agree with \cite[Proposition 4.5]{GKR18} when $\lambda=0$. {Moreover, one can lift \eqref{E:Htransform} to a non-linear partial differential equation in duality with \eqref{E:SPDE}.  Indeed,  define the measure-valued  function 
			$ \bar H: t \to \bar H_t(dx) = H_t(x) \mathds{1}_{\{x \geq 0\}} dx.$ Then, it follows from \eqref{E:Htransform} that 
			\begin{align*}
			\bar H_t(dx) 
			&= {\rm i } h(x-t) \mathds{1}_{\{x > t\}}dx +  \int_0^t \delta_0(dx - (t-s)) ( -  \lambda \langle \bar H_{s}, K \rangle + \frac{\nu^2}2 \langle \bar H_{s}, K \rangle^2)ds
			\end{align*}
			which can be seen as the mild formulation of the following partial differential equation
					\begin{equation}\label{E:pdemeasure}
					 d\bar H_t(dx) = (-\frac{d}{d x} \bar H_t(dx) + \delta_0(dx ) ( -  \lambda \langle \bar H_{t}, K \rangle + \frac{\nu^2}2 \langle \bar H_{t}, K \rangle^2)) dt, \;\;\, \bar H_0(dx)={\rm i} h(x) \mathds{1}_{\{x > t\}}dx . 
					\end{equation}
We refer to \cite{CT18b,CT18a} for similar results in the discontinuous setting.
		The previous results highlight not only the correspondence   between  {stochastic Volterra equations} of the form \eqref{E:HestonVg} and {stochastic partial differential equations} \eqref{E:SPDE} but also between their dual objects, that is the {Riccati-Volterra equation} \eqref{E:RicHeston1} and the non-linear {partial differential equation} \eqref{E:pdemeasure}. 
		One can establish a correspondence between \eqref{E:HestonVg} and other related {stochastic partial differential equations} which, unlike $(g_t)_{t \geq 0}$, do not necessarily  have a financial interpretation but  for which the dual object satisfies a nicer non-linear {partial differential equation} than \eqref{E:pdemeasure}, see \cite{mytnik2015uniqueness}. }

	\section{Application: {square-root process in Banach space}}\label{S:cm}
	As an application of Theorems \ref{T:VolterraHestongExistence}, \ref{T:VolterraHestong}, \ref{T:condLaw0}, we obtain conditions for weak existence and uniqueness of the following (possibly) infinite-dimensional system of stochastic differential equations
	\begin{align}\label{eq:Ygamma}
	dU_t(x) = \left(- x U_t(x)  -\lambda  \int_0^{\infty} U_t(z) {\mu}(dz)    \right) dt + \nu \sqrt{\int_0^{\infty} U_t({z}) {\mu}(dz) } dW_t,  \quad   x \in {\supp(\mu)},
	\end{align} 
	for a fixed positive measure of locally bounded variation  $\mu$\footnote{We  use the notation $\supp(\mu)$ to denote the   support of a measure $\mu$, that is the  set of all points {for which every  open neighborhood has a positive measure. Here we assume that the support is in $\R_+$.}}. {This} is achieved by linking \eqref{eq:Ygamma} to a stochastic Volterra equation of the form  \eqref{E:HestonVg}  {with the following kernel}
	\begin{align}\label{eq:laplace mu}
	K(t)= \int_0^{\infty} e^{-x t} \mu(dx), \quad  t>0 .
	\end{align}
	We will assume that $\mu$ is a positive measure of locally bounded variation such that 
	\begin{equation}   \label{mu_gamma}\tag{$H_2$}
\int_0^\infty (1\wedge (x h)^{-1/2}) \mu(dx) \leq C h^{(\gamma-1) / 2 }, \quad 	\int_0^\infty x^{-1/2}(1\wedge (x h)) \mu(dx) \leq C h^{\gamma / 2};   \quad h > 0, 
	\end{equation}
	for some $\gamma \in (0,2]$ and positive constant $C$. The reader may check that in that case $K$ satisfies \eqref{K_gamma}. Furthermore, \cite[Theorem 5.5.4]{GLS90} guarantees the existence of the resolvent of the first kind $L$ of $K$ and that \eqref{eq:K orthant}  is satisfied for the shifted kernels $\Delta_h K$ for any $h \in [0,1]$. Hence, $K$ satisfies assumptions of Theorems \ref{T:VolterraHestongExistence} and \ref{T:VolterraHestong}.\vs1

	By a solution {$U$} to \eqref{eq:Ygamma} we mean a family of continuous processes $(U(x))_{x \in \supp(\mu)}$ such that $x \to U_t(x) \in L^1(\mu)$ for any $t \geq 0$, $(\int_0^\infty U_t(x) \mu(dx))_{t \geq 0}$ is a continuous process and {\eqref{eq:Ygamma} holds a.s.~on some filtered probability space.}
	If {such solution exists}, we set {$V = \int_0^\infty U_{\cdot}(x) \mu(dx)$ and $g_0 = \int_0^\infty U_0(x) e^{-x(\cdot)} dx $.} Thanks to \eqref{mu_gamma},  the stochastic Fubini theorem yields for each $t \geq 0$ 
	\begin{equation} \label{sol1}
	V_t = g_0(t) + \int_0^t K(t-s) (-\lambda V_s ds + \nu \sqrt{V_s} dW_s){.}
	\end{equation}
	 The processes above being continuous, the equality holds in terms of processes. Thus, provided that $g_0$ belongs to $\Gc_K$, {Theorem \ref{T:VolterraHestong}} leads to the weak uniqueness of \eqref{eq:Ygamma} {because for each $x \in \supp(\mu)$,
\begin{equation}  \label{V_t_x}
 U_t(x) = e^{-xt} U_0(x) + \int_0^t e^{-x(t-s)}  (-\lambda V_s ds + \nu \sqrt{V_s} dW_s), \quad t \geq 0.
	\end{equation}}
	
	On the other hand, if we assume that $g_0 = \int_0^\infty U_0(x) e^{-x(\cdot)} \mu(dx) \in \Gc_K$ for some initial family of points $(U_0(x))_{x \in \supp(\mu)} \in L^1(\mu)$, there exists a continuous solution $V$ for \eqref{sol1} by Theorem \ref{T:VolterraHestongExistence}. In that case, we define for each $x \in \supp(\mu)$, the continuous process {$U(x)$ as in \eqref{V_t_x}.} Thanks to {\eqref{mu_gamma} and  \eqref{E:momentestimate}}, another application of the} stochastic Fubini theorem combined with the fact that $V$ satisfies \eqref{sol1} yields that, for each $t \geq 0$,  $(U_t(x))_{x \in \supp(\mu)} \in L^1(\mu)$ and 
\begin{equation}  \label{equality_V}
V_t = \int_0^\infty U_t(x) \mu(dx). 
\end{equation}
Moreover, by an integration by parts, we get for each $x \in \supp(\mu)$,
$$ U_t(x) =  e^{-xt}U_0(x)+ Z_t  e^{-x t} + \int_0^t x e^{-x(t-s)} (Z_s - Z_t) ds,  $$
with $Z = \int_0^\cdot (-\lambda V_s ds + \nu \sqrt{V_s} dB_s)$. {We know that for  fixed $T>0$, $\eta \in (0, 1/2)$ and for almost any $\omega \in \Omega$ there exists a positive constant $C_T(\omega)$ such that 
$ |Z_s - Z_t| \leq C_T(\omega) |t-s|^{\eta}$ for all   $t,s \in [0, T]$.} Hence for any $t \in [0,T]$ and $x \in \supp(\mu)$
$$ |U_t(x)| \leq |U_0(x)| + C_T(\omega) {e^{-xt}} t^\eta +C_T(\omega) x \int_0^t e^{-xs} s^{\eta} ds = |U_0(x)| +C_T(\omega) {\eta \int_0^t e^{-xs} s^{\eta-1} ds.}  $$
Then, 
$$ \sup_{t \in [0,T]}|U_t(x)| \leq  |U_0(x)| +C_T(\omega)  {\eta}\int_0^T e^{-xs} s^{\eta} ds \in L^1(\mu).  $$
Therefore by  dominated convergence {theorem}, the process $(\int_0^\infty U_t(x) {\mu(dx)})_{t \geq 0}$ is continuous. In particular, \eqref{equality_V} holds in terms of processes and it follows from \eqref{V_t_x} that $U$ is a solution of \eqref{eq:Ygamma}. 

This leads to the weak existence and uniqueness of \eqref{eq:Ygamma} if the initial family of points $(U_0(x))_{x \in \supp(\mu)}$ belongs to the following space ${\cal D}_\mu$ defined by
\begin{equation} \label{A_space}
{\cal D}_\mu = \{(u_x)_{x \in \supp(\mu)} \in L^1({\supp(\mu)}); \quad \int_0^\infty u_x e^{-x (\cdot)} \mu(dx) \in \Gc_K \},  
\end{equation}
with $K$ given by \eqref{eq:laplace mu}. Notice that for fixed $t_0 \geq 0$ and for any $t \geq 0$ and $x \in \supp(\mu)$,
$$ U_{t + t_0}(x) = U_{t_0}(x) e^{-xt} + \int_0^t e^{-x(t-s)} {\left(-\lambda \int_0^\infty U_{s+t_0}(z) \mu(dz) + \nu \sqrt{ \int_0^\infty {U}_{s+t_0}(z) \mu(dz)} dW_{s+{t_0}}\right) }$$
and then by stochastic Fubini theorem
$$ \int_0^\infty U_{t + t_0}(y) \mu(dy) = g_{t_0}(t) + \int_0^t {K(t-s)} {\left(-\lambda \int_0^\infty U_{s+t_0}(z) \mu(dz) + \nu \sqrt{ \int_0^\infty U_{s+t_0}(z) \mu(dz)} dW_{s+t_0}\right)}, $$
with $g_{t_0}(t) = \int_0^\infty U_{t_0}(y) e^{-yt} \mu(dy)$. Thanks to Theorem \ref{T:condLaw0}, we deduce that $g_{t_0} \in \Gc_K$ and therefore $({U}_{t_0}(x))_{x \in \supp(\mu)}$ belongs to ${\cal D}_\mu$. As a conclusion, the space ${\cal D}_\mu$ is stochastically invariant with respect to the family of processes $(U(x))_{x \in \supp(\mu)}$. 
\begin{theorem} Fix $\mu$ a positive measure of locally bounded variation satisfying \eqref{mu_gamma}.There exists a unique weak solution $U$ of \eqref{eq:Ygamma} for each initial family of points $(U_0(x))_{x \in \supp(\mu)} \in {\cal D}_\mu$. Furthermore for any $t \geq 0$, $ {{(U_t(x))_{x \in \supp(\mu)}}} \in  {\cal D}_\mu$.
\end{theorem}

\begin{table}[h!]
	\centering
	\begin{tabular}{c c c c }
		\hline
		& $K(t)$ & Parameter restrictions & $\mu(d\gamma)$ \\ 
		\hline
		Fractional		& $c\,\frac{t^{\alpha-1}}{\Gamma(\alpha)}$ & $\alpha \in (1/2,1)$ & $c\frac{x^{- \alpha}}{\Gamma(\alpha)\Gamma(1-\alpha)}dx$\\ \\
		Gamma		& $c{\rm e}^{-\lambda t} \frac{t^{\alpha-1}}{\Gamma(\alpha)} $ &   $\lambda \geq 0, \alpha \in (1/2,1)$  & $c\frac{(x-\lambda)^{-\alpha} \mathds{1}_{( \lambda, \infty)}(x)}{\Gamma(\alpha)\Gamma(1-\alpha)}dx $ \\\\
		Exponential sum	& $\displaystyle\sum_{i=1}^n c_i {\rm e}^{-\gamma_i t}$ & $c_i,\gamma_i \geq 0$ & $ \displaystyle \sum_{i=1}^n c_i \delta_{\gamma_i}(dx) $\\ 
		\hline
	\end{tabular}
	\caption{Some measures $\mu$ satisfying \eqref{mu_gamma} with their associated kernels $K$. Here $c\geq 0$.}
	\label{T:summarycm}
\end{table}

\begin{remark}[Representation of $V$ in terms of $U$]\label{R:repV} 
 In a similar fashion one can establish the existence and uniqueness of the following time-inhomogeneous version of \eqref{eq:Ygamma}
		\begin{align}\label{E:Ygamma2}
			dU_t(x) = \left(- x U_t(x)  -\lambda \left( g_0(t) + \langle 1, U_t\rangle_{\mu}  \right )   \right) dt + \nu \sqrt{g_0(t) +  \langle 1, U_t\rangle_{\mu} } dW_t, \quad x \in \supp(\mu),
			\end{align}
		whenever 		$$g_0 = {\widetilde{g_0} } + \int_0^{\infty} e^{- x (\cdot)} U_0(x) \mu(dx)    \in \Gc_K, $$ 
		with ${\widetilde{g_0}}: \R_+ \to \R$. In this case,  
		$$ {\widetilde{g_0}(t+\cdot)} +  \int_0^{\infty} e^{-x (\cdot)}  U_{t}(x) \mu(dx) \in \Gc_K, \quad  t \geq 0.$$
		In particular, for  $U_0 \equiv 0$, $g_0 ={\widetilde{g_0}} \in \Gc_K$ and $K$ as in \eqref{eq:laplace mu}, the  solution $V$ to the stochastic Volterra equation \eqref{E:HestonVg} and the  forward process $(g_t)_{t \geq 0}$  admit the following  representations 
			\begin{align}
			V_t=  g_0(t) + \langle 1, U_t \rangle_{\mu}, \quad g_{t_0}(t) = g_0(t_0+t) + \langle e^{-t(\cdot)}, U_t \rangle_{\mu}, \quad  t,t_0 \geq 0,\label{E:rep g} 
			\end{align}
		where we used the notation $\langle f,g\rangle_{\mu} = \int_0^t f(x)g(x) \mu(dx)$. These results are in the spirit of \cite{CC98,HS15}.
\end{remark}
When $\mu$ has finite support, \eqref{E:Ygamma2} is a  finite dimensional diffusion with an affine structure in the sense of \cite{dfs}. This underlying structure carries over to the case of infinite support and is the reason behind the tractability of the {Volterra Heston model}. 

\begin{remark}[Affine structure of $(\log S,V)$ in terms of $U$]  Let the {notations} and assumptions of Remark \ref{R:repV} be in force.
Relying  on the existence and uniqueness of the Riccati-Volterra equation \eqref{E:RicHeston1} one can establish the existence and uniqueness  of a differentiable (in time) solution $\chi_2$ to
 the following (possibly) infinite-dimensional system of Riccati ordinary differential equations
\begin{align}
\partial_t \chi_2(t, x) = - x \chi_2(t, x)  + F\left(\psi_1(t) ,   \langle  \chi_2(t,\cdot), 1 \rangle_{\mu}\right), \quad  \chi_2(0, x)=u_2,  \quad  x \in \supp \mu, \; t\geq 0,  \label{eq:Ric tilde psi }
\end{align}
such that {$\chi_2 (t, \cdot) \in L^{1}(\mu)$}, for all $t \geq 0$ and {$t \to \langle \chi_2(t, \cdot),1 \rangle_{\mu} \in L^2_{\rm loc}(\R_+)$} with $\psi_1$  given by \eqref{E:RicHeston0}   and $F$  by \eqref{E:phichi-2}. Moreover,  the unique global solution $\psi_2\in L^2_{\rm loc}(\R_+,\mathbb C^*)$ to the Riccati--Volterra equation \eqref{E:RicHeston1} admits the following representation 
\begin{align*}\label{eq:rep psi}
\psi_2 = \int_0^{\infty} \chi_2({\cdot}, x) \mu(dx),
\end{align*}	
where $\chi_2$ is the unique solution to \eqref{eq:Ric tilde psi }. In particular, combining the {equality} above with \eqref{E:charwithgt} and the representation of $(g_t)_{t \geq 0}$ in  \eqref{E:rep g}
leads to the exponentially-affine functional  
\begin{align*}
\E\left[ \exp\left( u X_T + (f*X)_T \right) \Mid \Fc_t \right] = \exp\left({\phi( t,T)  +  \psi_1(T-t) \log S_t  + \langle \chi_2(T-t, \cdot) , U_t\rangle_\mu }\right)
\end{align*}
 for all $t \leq T$ where  $\phi(t,T) =  (u_2 \Delta_{t} g_0 + F(\psi_1,\psi_2) * \Delta_{t}g_0)(T-t)$,
$(u,f)$  as in \eqref{E:condcoeffrealpart} and $U$ solves \eqref{E:Ygamma2}.
\end{remark}

{The representations of this section lead to a generic approximation of the Volterra Heston model by finite-dimensional affine diffusions, see \cite{AJEE18a} for the rigorous treatment of these approximations.}
\begin{appendices}

	\section{Existence results for stochastic Volterra equations}\label{A:sve}
	In this section, we  consider the following $d$-dimensional stochastic Volterra equation 
	\begin{equation} \label{generalDiffusion}
	X_t = g(t) + \int_0^t K(t-s) b(X_s) ds + \int_0^t K(t-s) \sigma(X_s) dW_s, 
	\end{equation}
	where $K \in {L}^2_{\rm loc} (\R, {\mathbb R}^{d \times d})$, $W$ is a $m$-dimensional Brownian motion, $g : \R^d \rightarrow \mathbb R^d$, $b: \mathbb R^d \rightarrow \mathbb R^d$, $\sigma : \mathbb R^d \rightarrow {\mathbb R}^{d \times m}$ are continuous {with linear growth.} 
By adapting the proofs of \cite[Appendix A]{ALP17} (there, {$g$ is constant}), we obtain the following existence results. Notice how the domain $\Gc_K$ defined in~\eqref{E:DomainG} enters in the construction of constrained solutions in Theorem~\ref{T:existence orthant} below.
	
	\begin{theorem}\label{T:existenceg}Under~\eqref{K_gamma}, assume that {$g \in \Hc^{\gamma/2}$.}
		\begin{enumerate}[(i)]
			\item \label{T:existenceg1}
			If $b$ and $\sigma$ are Lipschitz continuous, \eqref{generalDiffusion} admits a unique continuous strong solution $X$.
			\item \label{T:existenceg2}
		If $b$ and $\sigma$ are continuous {with linear growth} and $K$ admits a resolvent of the first kind $L$, then~\eqref{generalDiffusion} admits a continuous weak solution $X$.
		\end{enumerate}
		 {In both cases, $X$ is {locally} H\"older continuous  of any order strictly smaller than $\gamma/2$ and}  		
		 \begin{equation}\label{E:momentestimate2}
		 \sup_{t \leq T}\E [|X_t|^p] < \infty ,  \quad  p  > 0,  \quad T >0.
		 \end{equation}
\end{theorem}

\begin{theorem} \label{T:existence orthant}
		Assume that $d=m=1$ and  that the scalar kernel $K$   satisfies \eqref{K_gamma}-\eqref{eq:K orthant}. Assume also that $b$ and $\sigma$ are continuous {with linear growth} such that $b(0)\ge0$ and $\sigma(0)=0$. Then \eqref{generalDiffusion} admits a nonnegative continuous weak solution for any $g  \in \Gc_{K}$.
		\end{theorem}

	\begin{proof} Theorem~\ref{T:existenceg}\eqref{T:existenceg2} yields the existence of an unsconstrained continuous weak solution $X$ to the  following modified equation 
	$
	X_t = g(t) + \int_0^t K(t-s) b(X_s^+) ds + \int_0^t K(t-s) \sigma(X_s^+) dW_s.
	$
		As in the proof of  of~\cite[Theorem~3.5]{ALP17}, it suffices to prove the nonnegativity of $X$ under the stronger condition, that, for some fixed $n \in \mathbb N$, 
		\begin{equation} \label{eq:orth:bdry}
		\text{$x \leq n^{-1}$ implies $b(x)\ge0$ and $\sigma(x)=0$.}
		\end{equation}
	Set $Z=\int (b(X)dt + \sigma(X)dW)$ and $\tau=\inf\{t\ge0\colon X_t <0\}$. Since $g(0) \geq 0$, $\tau \geq 0$. On $\{\tau<\infty\}$,
	\begin{equation} \label{eq:orth:0}
	X_{\tau+h}=g(\tau + h) + (K*dZ)_{\tau+h} = g(\tau + h) + (\Delta_h K * dZ)_\tau + \int_0^h K(h-s) d Z_{\tau+s}, \quad  h \geq 0.
	\end{equation}
	Using Lemma~\ref{thelemma} and Remark~\ref{DeltaK} below, together with the fact that $X\geq 0$ on $[0,\tau]$,
	\begin{align*}
	g(\tau +h) + (\Delta_h K * dZ)_\tau &= g(\tau +h) + (\Delta_h K * L)(0)(X - g)(\tau) \\
	&\quad + (d(\Delta_h K * L) * X)_{\tau} - (d(\Delta_h K * L) *g) (\tau)\\
	&\geq g(\tau+h) - (d(\Delta_hK*L)*g)(\tau) - (\Delta_hK * L)(0) g(\tau),
	\end{align*}
	which is nonnegative. In view of \eqref{eq:orth:0} it follows that
	\begin{equation} \label{eq:orth:3}
	X_{\tau+h} \ge
	\int_0^h K(h-s) \left( b(X_{\tau+s})ds + \sigma(X_{\tau+s})dW_{\tau+s} \right)
	\end{equation}
	on $\{\tau<\infty\}$ for all $h\ge0$. Now, on $\{\tau<\infty\}$, $X_{\tau}=0$ and $X_{\tau+h}<0$ for arbitrarily small $h$. On the other hand, by continuity there is some $\varepsilon>0$  such that $X_{\tau+h}\le n^{-1}$ for all $h\in[0,\varepsilon)$. Thus \eqref{eq:orth:bdry} and \eqref{eq:orth:3} yield $X_{\tau+h} \ge0$ for all $h\in[0,\varepsilon)$. This shows that $\tau=\infty$,  ending the proof.
 	\end{proof}
		 
	\section{Reminder on stochastic convolutions and resolvents}\label{A:conv}
	For a measurable function $K$ on $\R_+$ and a measure $L$ on $\R_+$ of locally bounded variation, the convolutions $K*L$ and $L*K$ are defined by
	\vspace{-0.3cm}
	$$
	\vspace{-0.3cm}
	(K*L)(t) = \int_{[0,t]} K(t-s)L(ds), \qquad (L*K)(t) = \int_{[0,t]} L(ds)K(t-s)
	$$
	whenever these expressions are well-defined. If $F$ is a function on $\R_+$, we write $K*F=K*(Fdt)$.
	We can show that $L * F$ is almost everywhere well-defined and belongs to $ L^p_{\rm loc}(\R_+)$, whenever $F \in L^p_{\rm loc}(\R_+)$. Moreover, $(F * G) * L = F * (G * L) $ {\em a.e.}, whenever $F, G \in L^1_{loc}(\R_+)$, see~\cite[Theorem~3.6.1 and Corollary~3.6.2]{GLS90} for further details. \vs1 
	
	For any continuous semimartingale $M = \int _0^. b_s ds + \int_0^. a_s dB_s$  the convolution
	$ (K*dM)_t = \int_0^t K(t-s)dM_s$
	 is well-defined as an It\^o integral for every $t\ge0$ such that 
	$\int_0^t |K(t-s)| |b_s| ds + \int_0^t |K(t-s)|^2 |a_s|^2 ds <\infty.$
	By stochastic Fubini Theorem,  see \cite[Lemma~2.1]{ALP17}, we have {$(L*(K*dM)) = ((L * K) * dM), \,  a.s.$} whenever $K \in L^2_{loc}(\R_+,\R)$ and $a, b$ are locally bounded {\em a.s}.\vs1
	
	We define the {\em resolvent of the first kind} of a $d\times d$-matrix valued kernel $K$, as the $\R^{d\times d}$-valued measure $L$ on $\R_+$ of locally bounded variation such that
	$K*L = L*K \equiv \id,$
	where $\id$ stands for the identity matrix, see~\cite[Definition~5.5.1]{GLS90}. The {\em resolvent of the first kind} does not always exist. The following results are shown in~\cite[Lemma~2.6]{ALP17}.  
	
	\begin{lemma} \label{L:ZX}
		Let $K\in L^2_{\rm loc}(\R_+)$ and $Z=\int_0^. b_s ds + \int_0^. \sigma_s dW_s$ a continuous semimartingale with $b$ and $\sigma$ locally bounded. Assume that $X$ and $K*dZ$ are continuous processes and that $K$ admits a resolvent of the first kind $L$. Then 
		{$
		X = K*dZ$ if and only if $  L*X=Z.
		$}
	\end{lemma}
	
	\begin{lemma}\label{thelemma} Assume that $K \in L^1_{\rm loc}(\R_+)$ admits a resolvent of the first kind $L$. For any $F \in L^1_{loc}(\R_+)$ such that $ F * L $ is right-continuous and of locally bounded variation one has 
	\vspace{-0.1cm}
		$$ \vspace{-0.1cm}
		F = (F*L)(0) K+ d(F*L) * K . $$
	\end{lemma}
	\begin{remark} \label{DeltaK}
		The previous lemma will be used with $F = \Delta_h K$, for a fixed $h \geq 0$. If $K$ is continuous on $(0, \infty)$, then $\Delta_h K*L$ is right-continuous. Moreover, if $K$ is  nonnegative and  $L$ is non-increasing  in the sense that  $s \to  L([s,s+t])$ is non-increasing for all $t\ge0$,  then $\Delta_h K *L$  is non-decreasing  {since $ \Delta_h K * L = 1 -\int_{(0,h]} K(h-s) L(\cdot + ds),  \,  t \geq 0. $}
			\end{remark}

 \end{appendices}
 
 \bibliographystyle{abbrv}
\setlength{\bibsep}{0pt plus 0.3ex}
  {\footnotesize\bibliography{bibAJEEMarkovianshort}}
 
\end{document}